\documentclass{article}
\usepackage[utf8]{inputenc}
\usepackage{authblk}
\usepackage{setspace}
\usepackage[margin=1in]{geometry}
\usepackage{graphicx}
\graphicspath{ {./figures/} }
\usepackage{subcaption}
\usepackage{amsmath}
\usepackage{lineno}
\usepackage{comment}
\usepackage{amsthm}

\newtheorem{thm}{Theorem}[section]

\newtheorem{lem}[thm]{Lemma}

\usepackage{amsfonts} %For \mathbb{R}
\usepackage{mathtools} %Declare pair deliminator

\usepackage{url}
\usepackage{hyperref}
\usepackage{indentfirst}
\usepackage{xcolor}
\usepackage{float}

\hypersetup{
    colorlinks,
    linkcolor={red!50!black},
    citecolor={blue!50!black},
    urlcolor={blue!80!black}
}
\usepackage[natbib=true,backend=biber,sorting=none,style=ieee, citestyle=numeric-comp]{biblatex}

\title{Dynamics and Control of Additional Food Provided Prey-Predator Systems exhibiting Holling Type-III Functional Response and Intra-specific Competition among Predators}

\author[1]{D Bhanu Prakash}
\author[2]{D K K Vamsi}

\affil[1, 2]{ \ Department of Mathematics and Computer Science, Sri Sathya Sai Institute of Higher Learning, India.}
\affil[2]{Center for Excellence in Mathematical Biology, Sri Sathya Sai Institute of Higher Learning, India.}
\affil[1]{Corresponding Author. Email: dbhanuprakash@sssihl.edu.in}

\date{}

\begin{document}

\maketitle

%%%%%% Abstract %%%%%%
\begin{abstract} {{
\noindent The dynamics of predator-prey systems influenced by intra-specific competition and additional food resources have increasingly become a subject of rigorous study in the realm of mathematical biology. In this study, we consider an additional food provided prey-predator model exhibiting Holling type-III functional response and the intra-specific competition among predators. We prove the existence and uniqueness of global positive solutions for the proposed model. We study the existence and stability of equilibrium points and further explore the possible bifurcations. We numerically depict the presence of Hysteresis loop in the system. We further study the global dynamics of the system and discuss the consequences of providing additional food. Later, we do the time-optimal control studies with respect to the quality and quantity of additional food as control variables by transforming the independent variable in the control system. We show that the findings of these dynamics and control studies emphasises the role of additional food and intra-specific competition in bio-control of pests.
}}
\end{abstract}

{ \bf {keywords:} } Prey-Predator System; Holling type-III response; Additional Food; Intra-specific Competition; Bifurcation; Hysteresis loop; Time-optimal control; Pest management;


%\linenumbers

\section{Introduction} \label{Intro}

\indent Papaya is an important tropical fruit crop in India, but its production is often threatened by various pests. One such major pest is the papaya mealybug, which has caused significant crop loss. To combat this, biocontrol agents like \textit{Acerophagus papayae}, a natural enemy of the pest was imported from Puerto Rico. This natural enemy for papaya was first found out in Mexico. This natural enemy have been introduced into the vast papaya farms in Coimbatore in 2008, with notable success \cite{BioControl_India}. This method of using one organism to control another is gaining popularity due to its environmentally friendly nature and reduced need for chemical pesticides. Biocontrol offers sustainable pest management by minimizing harm to non-target species, preserving ecological balance and reducing chemical residues in food and soil. However, despite its benefits, biocontrol can lead to unintended outcomes such as the overpopulation or extinction of certain species, if not carefully planned and analyzed.

Mathematical models are essential tools in understanding and predicting the outcomes of pest management strategies, including biocontrol. In such models, the pest can be treated as the “prey” and the biocontrol agent as the “predator,” allowing researchers to apply mathematical theories of prey-predator system. Analyzing these systems helps identify conditions for long-term coexistence, extinction, or sudden changes in population levels.

The major component of the prey-predator models is the functional response. Among various possible responses, we study the type-III responses, with their sigmoidal curve. These are often seen in systems where predators have a learning phase or where prey exhibit low density refuge effects. Recently, authors in \cite{V3EarthSystems,V3JTB,V4Acta,V4DEDS} studied the systems exhibiting Holling type-III and Holling type-IV functional responses. 

However, most of these models reach infinite prey in the absence of predator. Also, these models does not incorporate the intra-specific competition among predators. As the ecosystems are of limited resources, competition plays an important role in the dynamics of the system. Bazykin models were the first few models where the effect of intra-specific competition is included \cite{bazykin1976structural,Adhikary2021}. 

 In this study, we define the additional food provided prey-predator system exhibiting Holling type-III functional response and intra-specific competition among predators. Further, we present the stability, bistability, bifurcation studies on the proposed system. Moreover, time optimal control studies help determine the most efficient strategies for introducing biocontrol agents to reduce pest populations quickly and sustainably. These mathematical insights are vital to ensuring biocontrol efforts remain effective and ecologically safe.

The article is structured as follows: Section \ref{sec:3iscdmodel} introduces the prey-predator model with intra-specific competition and provision of additional food among predators. Section \ref{sec:3iscdposi} proves the positivity and boundedness of the solutions of the proposed system. Section \ref{sec:3iscdequil} investigates the conditions for the existence of various equilibria. The local stability of these equilibria is presented in section \ref{sec:3iscdstab}. In section \ref{sec:3iscdbifur}, we present the various possible local bifurcations exhibited by the proposed model both analytically and numerically. Section \ref{sec:3iscdglobaldynamics} studies the global dynamics of the proposed system in the parameter space of quality and quantity of additional food. Section \ref{sec:3iscdconseq} provides a detailed analysis of the consequences of providing additional food. Section \ref{sec:3iscdtimecontrol} presents the study on time-optimal control problems with quality or quantity of additional food as control parameters. Finally, we present the discussions and conclusions in section \ref{sec:disc}.

\section{Model Formulation} \label{sec:3iscdmodel}

\indent The Prey-Predator model in the presence of Holling type-III predator functional response, also known as the p = 2 S-type functional response and intra-specific competition among predators is given by
\begin{equation} \label{3iscdnoa}
	\begin{split}
		\frac{\mathrm{d} N}{\mathrm{d} T} & = r N \left(1-\frac{N}{K} \right)-\frac{c N^2 P}{a^2+N^2}, \\
		\frac{\mathrm{d} P}{\mathrm{d} T} & = \delta_1 \left( \frac{N^2}{a^2+N^2} \right) P - m_1 P - d P^2.
	\end{split}
\end{equation}

After provision of additional food, the model gets transformed to the following. In this case, we use functional responses from \cite{V3JTB}.
\begin{equation} \label{3iscda}
	\begin{split}
		\frac{\mathrm{d} N}{\mathrm{d} T} & = r N \left(1-\frac{N}{K} \right)-\frac{c N^2 P}{a^2+N^2+\alpha \eta  A^2}, \\
		\frac{\mathrm{d} P}{\mathrm{d} T} & = \delta_1 \left( \frac{N^2 + \eta A^2}{a^2+N^2+\alpha \eta A^2} \right) P - m_1 P - d P^2.
	\end{split}
\end{equation}

Here, the parameter $\delta_1$ denotes the conversion efficiency that represents the rate at which  prey biomass gets converted into predator biomass and can be obtained as a ratio of nutritive value of prey upon the handling time of predator.  The term $\alpha$ denotes the ratio between the maximum growth rates of the predator when it consumes the prey and additional food respectively. This term can be seen to be an equivalent of {\textit {\bf{quality}} }of additional food. The term $\eta$ represents the ratio between the search rate of the predator for additional food and prey respectively. The term $-dP^2(t)$ accounts for the intra-specific competition among the predators in order to avoid their unbounded growth in the absence of target prey. The biological meaning for all the parameters involved in the systems (\ref{3iscdnoa}) and (\ref{3iscda}) are enlisted and described in \autoref{3iscdparam_tab}. Further details regarding the derivation of the functional response, model formulation and the parameters can be found in \cite{V3JTB}.

To lessen the complexity in the analysis, we now reduce the number of parameters in the model (\ref{3iscda}) by introducing the transformations $N = a x,\  P=\frac{ar y}{c},\ t = r T$. The system (\ref{3iscda}) gets transformed to:
\begin{equation} \label{3iscd}
	\begin{split}
		\frac{\mathrm{d} x}{\mathrm{d} t} & = x \left( 1 - \frac{x}{\gamma} \right) - \frac{x^2 y}{1 + x^2 + \alpha \xi }, \\
		\frac{\mathrm{d} y}{\mathrm{d} t} & = \frac{\delta \left(x^2 + \xi \right) y}{1+x^2+\alpha \xi} - m y - \epsilon y^2,
	\end{split}
\end{equation}
where $ \gamma = \frac{K}{a},\  \xi = \eta (\frac{A}{a})^2,\  \epsilon = \frac{d a}{c},\ m = \frac{m_1}{r},\ \delta = \frac{\delta_1}{r}$. Here the term $\xi$ denotes
the quantity of additional food perceptible to the predator with respect to the prey relative to the nutritional value of prey to the additional food. Hence this can be seen to be an equivalent of {\textit {\bf{quantity}} }of additional food.

\begin{table}[bht!]
	\centering
	\begin{tabular}{ccc}
		\hline
		Parameter & Definition & Dimension \\  
		\hline
		T & Time & time\\ 
		N & Prey density & biomass \\
		P & Predator density & biomass \\
		A & Additional food & biomass \\
		r & Prey intrinsic growth rate & time$^{-1}$ \\
		K & Prey carrying capacity & biomass \\
		c & Rate of predation & time$^{-1}$ \\
		a & Half Saturation value of the predators & biomass \\
		$\delta$ & Conversion efficiency & time$^{-1}$ \\
		m & death rate of predators in absence of prey & time$^{-1}$ \\
		d & Predator Intra-specific competition & biomass$^{-1}$ time$^{-1}$ \\
		$\alpha$ & quality of additional food & Dimensionless \\
		$\xi$ & quantity of additional food & biomass$^{2}$ \\
		\hline
	\end{tabular}
	\caption{Description of variables and parameters present in the systems (\ref{3iscdnoa}) and (\ref{3iscda}).}
	\label{3iscdparam_tab}
\end{table}

\section{Positivity and boundedness of the solution} \label{sec:3iscdposi}

\subsection{Positivity of the solution}

In this section, we demonstrate that the positive $xy$-quadrant is an invariant region for the system (\ref{3iscd}). Specifically, this means that if the initial populations of both prey and predator start in the positive $xy$-quadrant (i.e., $x(0) > 0$ and $y(0) > 0$), they will remain within this quadrant for all future times.

If prey population goes to zero (i.e., $x(t)=0$), then it is observed from the model equations (\ref{3iscd}) that $\frac{\mathrm{d} x}{\mathrm{d} t} = 0.$ This means that the prey population is constant (remains at zero) and cannot be negative. This holds even for the case when predator population goes to zero (i.e., $y=0$). Notably, $x=0$ and $y=0$ serve as invariant manifolds, with $\frac{\mathrm{d} x}{\mathrm{d} t} \Big|_{x=0} = 0$ and $\frac{\mathrm{d} y}{\mathrm{d} t} \Big|_{y=0} = 0$. Therefore, if a solution initiates within the confines of the positive $xy$-quadrant, it will either remains positive or stays at zero eternally (i.e., $x(t) \geq 0$ and $y(t) \geq 0 \ \forall t>0$ if $x(0)>0$ and $y(0)>0$). 

\subsection{Boundedness of the solution}

\begin{thm}
	Every solution of the system (\ref{3iscd}) that starts within the positive quadrant of the state space remains bounded. \label{3iscdbound}
\end{thm} 

\begin{proof}
	We define $W = x + \frac{1}{\delta}y$. Now, for any $K > 0$, we consider,
	\begin{equation*}
		\begin{split}
			\frac{\mathrm{d} W}{\mathrm{d} t} + K W &= \frac{\mathrm{d} x}{\mathrm{d} t} + \frac{1}{\delta} \frac{\mathrm{d} y}{\mathrm{d} t} + K x + \frac{K}{\delta} y \\
			&=  x \left( 1 - \frac{x}{\gamma} \right) - \frac{x^2 y}{1 + x^2 + \alpha \xi } \\
			& \ \ +\frac{1}{\delta}  \left( \frac{\delta y \left(x^2 + \xi \right)}{1+x^2+\alpha \xi} - m y - \epsilon y^2 \right) + K x + \frac{K}{\delta} y \\
			&= x -\frac{x^2}{\gamma} - \frac{x^2y}{1+x^2+\alpha \xi} +  \left( \frac{x^2 + \xi}{1 + x^2 + \alpha \xi} \right) y \\  & \ \  - \frac{m}{\delta} y - \frac{\epsilon}{\delta} y^2+ K x + \frac{K}{\delta} y \\
			&= (1+K) x -\frac{x^2}{\gamma} + \frac{\xi y}{1 + x^2 + \alpha \xi} + \frac{K-m}{\delta} y  - \frac{\epsilon}{\delta} y^2  \\
			& \leq \frac{\gamma (1+K)^2}{4} +  \frac{\xi}{\epsilon } + \frac{(K-m)^2}{4 \epsilon} = M (,say)\\
		\end{split}
	\end{equation*}
	$$\frac{\mathrm{d} W}{\mathrm{d} t} + K W \leq M.$$
	Using Gronwall's inequality \cite{howard1998gronwall}, we now find an upper bound on $W(t)$. 
	
	This inequality is in the standard linear first-order form, and we solve it by multiplying both sides by an integrating factor $e^{Kt}$. This simplify the above inequality:
	$$\frac{\mathrm{d}}{\mathrm{d} t} (W(t) e^{Kt}) \leq M e^{Kt}.$$
	Now, integrating both sides from $0$ to $t$, we get
	$$0 \leq W(t) \leq \frac{M}{K} (1 - e^{-Kt}) + W(0) e^{-Kt}.$$
	Therefore, $0 < W(t) \leq \frac{M}{K}$ as $t \rightarrow \infty$. This demonstrates that the solutions of system (\ref{3iscd}) are ultimately bounded, thereby proving Theorem \autoref{3iscdbound}.
\end{proof}

The \text{Picard-Lindel\"of theorem} guarentees the existence of a unique solution that exists locally in time for the system (\ref{3iscd}), given any initial conditions $x(0) = x_0 > 0$ and $y(0) = y_0 > 0$. This happens because the RHS terms in (\ref{3iscd}) are continuous and locally Lipschitz. Since the solution does not blow up in finite time (i.e., that the solution exists for all $t \geq 0$), global existence is also guaranteed.

\section{Existence of Equilibria} \label{sec:3iscdequil}

In this section, we investigate the existence of various equilibria that system (\ref{3iscd}) admits and study their stability nature. We first discuss the nature of nullclines of the considered system and the asymptotic behavior of its trajectories. We consider the biologically feasible parametric constraint $\delta > m$. 

The prey nullclines of the system (\ref{3iscd}) are given by 
$$ x = 0 , \  \ 1- \frac{x}{\gamma}  - \frac{x y}{1+x^2 + \alpha \xi} = 0. $$

The predator nullclines of the system (\ref{3iscd}) are given by
$$ y = 0, \ \ \frac{\delta (x^2 + \xi)}{1+x^2+ \alpha \xi} - m - \epsilon y = 0. $$

Upon simplification, the non trivial predator nullcline is given by
\begin{equation} \label{3iscdpredator}
	y = \frac{(\delta - m) x^2 + \delta \xi - m (1+ \alpha \xi)}{\epsilon (1+x^2+\alpha \xi)}.
\end{equation}

This nullcline touches $y$-axis only at $\left(0,\frac{\delta \xi - m (1+\alpha \xi)}{\epsilon (1+\alpha \xi)}\right)$. In the absence of additional food, the nullcline touches negative $y$-axis i.e., $(0,\frac{-m}{\epsilon})$. With the provision of additional food, the predator nullcline moves upwards. It touches the positive $y$-axis only when $ \delta \xi - m (1+\alpha \xi) \geq 0$. This nullcline intersects with the positive $x$-axis at $(\sqrt{\frac{\delta \xi - m(1+\alpha \xi)}{m - \delta}},0)$ only when $\delta \xi - m(1+\alpha \xi) \leq 0$. 

Thus, the nullcline intersects only the positive $x$-axis and not the positive $y$-axis when $\delta \xi - m(1+\alpha \xi) \leq 0$, as shown in frames A and C of \autoref{3iscdiso}. Also the nullcline intersects only the positive $y$-axis and not the positive $x$-axis when $\delta \xi - m(1+\alpha \xi) \geq 0$, as shown in frames B and D of \autoref{3iscdiso}. 

Upon simplification, the non trivial prey nullcline is given by 
\begin{equation} \label{3iscdprey}
	y = \frac{\left(1-\frac{x}{\gamma}\right)(1+x^2+\alpha \xi)}{x}.
\end{equation} 

This prey nullcline passes through the point $(\gamma, 0)$ in the positive quadrant, with the $y$-axis ($x = 0$) acting as its asymptote. It lies in the positive $xy$-quadrant only when 
\begin{equation} \label{3iscdcond1}
	0 < x < \gamma.
\end{equation}
and remains negative outside this range. Therefore, (\ref{3iscdcond1}) becomes a necessary condition for the existence of interior equilibrium. 

The slope of the nullcline is
$$\frac{\mathrm{d} y}{\mathrm{d} x} = \frac{- \frac{2 x^3}{\gamma} + x^2 - 1 - \alpha \xi}{x^2}.$$

The slope of prey nullcline is a cubic equation in $x$ with the discriminant given by $$\triangle = \frac{4}{\gamma^2} (1 + \alpha \xi ) \left(\gamma^2 - 27 (1+\alpha \xi)\right).$$

The prey nullcline exhibits a crest and trough as long as $\gamma > 3 \sqrt{3 (1+\alpha \xi)}$ (i.e., $\triangle > 0$), as shown in frames A and B of \autoref{3iscdiso}. It will be monotonically decreasing otherwise (i.e., $\triangle < 0$), as shown in frames C and D of \autoref{3iscdiso}.

\begin{figure}[ht]
	\centering
	\includegraphics[width=\textwidth]{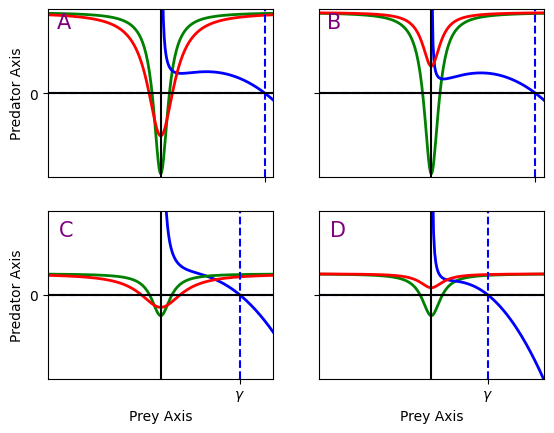}
	\caption{The possible configurations for the prey and predator nullclines of the system (\ref{3iscd}).}
	\label{3iscdiso}
\end{figure}

The possible configurations for the prey and predator nullclines are presented in \autoref{3iscdiso}. In this figure, the solid blue and red lines represent the prey and predator nullclines respectively. The solid green line represents the predator nullcline in the absence of additional food. From the qualitative theory of nullclines, it is observed that the system admits interior equilibrium in the case A and C only if the point of intersection of the predator nullcline is less than $\gamma$. When $\delta \xi - m(1+\alpha \xi) \leq 0$, $\sqrt{\frac{\delta \xi - m(1+\alpha \xi)}{m - \delta}} < \gamma \implies 0 > \delta \xi - m(1+\alpha \xi) > - \gamma^2 (\delta - m)$.  When $\delta \xi - m(1+\alpha \xi) > 0$, the predator nullcline never touches positive $x$ axis. However, in both cases, condition (\ref{3iscdcond1}) holds for the interior equilibrium. 

The system (\ref{3iscd}) always admits the points $E_0 = (0,0)$ and $E_1 = (\gamma , 0)$ as their trivial and axial equilibrium respectively. It also admits another axial equilibrium $E_2 = \left(0,\frac{\delta \xi - m (1+\alpha \xi)}{\epsilon (1+\alpha \xi)}\right)$ if $\delta \xi - m(1+\alpha \xi) > 0$. In the absence of additional food, this axial equilibrium $E_2 = (0,\frac{-m}{\epsilon})$ does not exist in the positive $xy$-quadrant.

The interior equilibrium of the system (\ref{3iscd}), if exists, is given by $E^* = (x^*, y^*)$  where 

\begin{equation} \label{3iscdystar}
	y^* =  \frac{(\delta - m) (x^*)^2 + \delta \xi - m (1+\alpha \xi)}{\epsilon (1 + (x^*)^2 + \alpha \xi )},
\end{equation}

and $x^*$ should satisfy the following cubic equation
\begin{equation} \label{3iscdxstar}
	\begin{split}
		\frac{\epsilon}{\gamma} (x^*)^5 - \epsilon (x^*)^4 + \left( \delta - m + \frac{2 \epsilon}{\gamma} (1 + \alpha \xi) \right) (x^*)^3 &  \\
		- 2 \epsilon (1+\alpha \xi) (x^*)^2 + \left(\frac{\epsilon}{\gamma} (1+\alpha \xi)^2 + \delta \xi - m (1+\alpha \xi) \right) (x^*) - \epsilon (1+\alpha \xi)^2 &= 0.
	\end{split}
\end{equation}

These conditions can be summarized as follows:
\begin{lem}\label{3iscdintcond}
	The system (\ref{3iscd}) admits an interior equilibrium $E^* = (x^*,y^*)$ which is a solution of the equations (\ref{3iscdystar}) and (\ref{3iscdxstar}) that satisfies the conditions  $\delta \xi - m(1+\alpha \xi) > - \gamma^2 (\delta - m)$ and $\frac{\epsilon}{1 + \frac{\epsilon}{\gamma}} < x^* < \gamma$.
\end{lem}

When additional food is provided, the system (\ref{3iscd}) can have atmost $5$ interior equilibria. To get more insights into the dynamics of these prey-predator models, we perform the stability analysis in the following section.

\section{Stability of Equilibria} \label{sec:3iscdstab}

In order to obtain the asymptotic behavior of the trajectories of the system (\ref{3iscd}), the associated Jacobian matrix is given by 

$$J = \begin{bmatrix}
	\frac{\partial}{\partial x} f(x,y)  & \frac{\partial}{\partial y} f(x,y)\\
	\frac{\partial}{\partial x} g(x,y) & \frac{\partial}{\partial y} g(x,y)
\end{bmatrix},$$

where

\begin{eqnarray*}
	f(x,y) &=& x \left(1-\frac{x}{\gamma} \right)- \frac{x^2y}{1+x^2+\alpha \xi}, \\
	g(x,y) &=& \delta \left( \frac{x^2+\xi}{1 + x^2 + \alpha \xi} \right) y - m y - \epsilon y^2,
\end{eqnarray*}

and 
\begin{eqnarray*}
	\frac{\partial}{\partial x} f(x,y) &=& 1 - \frac{2 x}{\gamma} - \frac{2 x y (1 + \alpha \xi)}{(1 + x^2 + \alpha \xi)^2}, \\
	\frac{\partial}{\partial y} f(x,y) &=& -\frac{x^2}{1 + x^2 + \alpha \xi }, \\
	\frac{\partial}{\partial x} g(x,y) &=& \frac{2 \delta x y (1+(\alpha - 1) \xi)}{(1+x^2 + \alpha \xi)^2}, \\
	\frac{\partial}{\partial y} g(x,y) &=& \frac{\delta (x^2 + \xi)}{1+x^2+\alpha \xi} - m - 2 \epsilon y.
\end{eqnarray*}

At the trivial equilibrium $E_0 = (0,0)$, we obtain the jacobian as 

\begin{equation*}
	J\left( E_0 \right) = \begin{bmatrix}
		1  & 0 \\
		0 & \frac{\delta \xi - m (1+ \alpha \xi)}{1+\alpha \xi}
	\end{bmatrix}.
\end{equation*}

The eigenvalues of this jacobian matrix are $1, \frac{\delta \xi - m (1+ \alpha \xi)}{1+\alpha \xi}$. If $\delta \xi - m (1+ \alpha \xi) > 0$, then both the eigenvalues have same signs. This makes the equilibrium $E_0 = (0,0)$ unstable. If $\delta \xi - m (1+ \alpha \xi) < 0$, then both the eigenvalues will have opposite signs. This makes the point $E_0 = (0,0)$ a saddle point. In the absence of additional food, $E_0 = (0,0)$ is a saddle point. 

At the axial equilibrium $E_1 = (\gamma ,0)$, we obtain the jacobian as 

\begin{equation*}
	J(E_1) = \begin{bmatrix}
		-1  & \frac{-\gamma ^2}{1+ \gamma ^2 + \alpha \xi} \\
		0 & \frac{(\delta - m) \gamma ^2 + \delta \xi - m (1+\alpha \xi)}{\gamma ^2 + 1 + \alpha \xi}
	\end{bmatrix}.
\end{equation*}

The eigenvalues of this jacobian matrix are $-1, \frac{(\delta - m) \gamma ^2 + \delta \xi - m (1+\alpha \xi)}{\gamma ^2 + 1 + \alpha \xi}$. If $(\delta - m) \gamma ^2 + \delta \xi - m (1+\alpha \xi) > 0$, then both the eigenvalues will have opposite sign resulting in a saddle point. If $(\delta - m) \gamma ^2 + \delta \xi - m (1+\alpha \xi) < 0$, then it will be an asymptotically stable node. In the absence of additional food, $E_1 = (\gamma, 0)$ is a saddle point if $(\delta - m) \gamma ^2 > m $. Else it is a stable equilibrium. 

We now consider another axial equilibrium which exists only for the additional food provided system (\ref{3iscd}). At this axial equilibrium $E_2 = \left(0, \frac{\delta \xi - m (1+\alpha \xi)}{\epsilon (1 + \alpha \xi)}\right)$, the associated jacobian matrix is given as 

\begin{equation*}
	J(E_2) = \begin{bmatrix}
		1  & 0 \\
		0 & -\frac{\delta \xi - m (1+\alpha \xi)}{1 + \alpha \xi} 
	\end{bmatrix}.
\end{equation*}

Since this equilibrium exists in positive $xy$-quadrant only when $\delta \xi - m (1+\alpha \xi) > 0$, the eigenvalues of the associated jacobian matrix are of opposite side which resulting in a saddle equilibrium. 

The following lemmas present the stability nature of the trivial and axial equilibria. 

\begin{lem}
	The trivial equilibrium $E_0 = (0,0)$ is saddle (unstable node) if 
	\begin{equation*}
		\delta \xi - m (1+\alpha \xi) < (>)\  0.
	\end{equation*}
\end{lem}

\begin{lem}
	The axial equilibrium $E_1 = (\gamma,0)$ is stable node (saddle) if 
	\begin{equation*}
		\delta \xi - m (1+\alpha \xi) < (>)\  - (\delta - m) \gamma ^2.
	\end{equation*}
\end{lem}

\begin{lem}
	The axial equilibrium $E_2 = \left(0, \frac{\delta \xi - m (1+\alpha \xi)}{\epsilon (1 + \alpha \xi)}\right) $ exists in positive $xy$-quadrant and is saddle if 
	\begin{equation*}
		\delta \xi - m (1+\alpha \xi) >\  0.
	\end{equation*}
\end{lem}

\subsection{Stability of Interior Equilibrium}

The interior equilibrium $E^* = (x^*,y^*)$ is the solution of system of equations  (\ref{3iscdystar}) and (\ref{3iscdxstar}) and satisfying the conditions in Lemma \ref{3iscdintcond}.

At this co existing equilibrium $E^* = (x^*, y^*)$, we obtain the jacobian as 

$$ J(E^*) = \begin{bmatrix}
	\frac{\partial}{\partial x} f(x,y)  & \frac{\partial}{\partial y} f(x,y) \\
	\frac{\partial}{\partial x} g(x,y) & \frac{\partial}{\partial y} g(x,y)
\end{bmatrix} \Bigg|_{(x^*,y^*)} . $$

The associated characteristic equation is given by

\begin{equation}
	\lambda ^2 - \text{Tr } J \bigg|_{(x^*,y^*)} \lambda + \text{Det } J \bigg|_{(x^*,y^*)} = 0.
\end{equation}

%%% Codes - > 3iscd2.ipynb
Now 

\begin{eqnarray*}
	& & \text{Det } J \bigg|_{(x^*,y^*)}  \\
	&=& \left( \frac{\partial f}{\partial x} \frac{\partial g}{\partial y} - \frac{\partial f}{\partial y} \frac{\partial g}{\partial x} \right)  \bigg|_{(\bar{x},\bar{y})} \\
	&=& \left(1 - \frac{2 x}{\gamma} - \frac{2 x y (1 + \alpha \xi)}{(1 + x^2 + \alpha \xi)^2}\right) \left(\frac{\delta (x^2 + \xi)}{1+x^2+\alpha \xi} - m - 2 \epsilon y \right) \\ 
	& & - \left(-\frac{x^2}{1 + x^2 + \alpha \xi }\right) \left(\frac{2 \delta x y (1+(\alpha - 1) \xi)}{(1+x^2 + \alpha \xi)^2}\right) \bigg|_{(x^*,y^*)}
\end{eqnarray*}

Upon simplification, we have
\begin{equation}  \label{3iscdintdet}
	\begin{split}
		\text{Det } J \bigg|_{(x^*,y^*)} = &\  \frac{2 \delta x^3 y (1 +\alpha \xi - \xi)}{(1+x^2+\alpha \xi)^3} + \epsilon x y \left(\frac{1}{\gamma} + \frac{(1+\alpha \xi - x^2) y}{(1+\alpha \xi + x^2)^2}\right)\bigg|_{(x^*,y^*)}.
	\end{split}
\end{equation} 

The determinent of the jacobian is positive when $1 + \alpha \xi - \xi > 0$ and ${x^*}^2 < 1 + \alpha \xi$.

The trace of the jacobian matrix is given by 
\begin{eqnarray*}
	\text{Tr } J \bigg|_{(x^*,y^*)}  &=& \frac{\partial f}{\partial x} \bigg|_{(x^*,y^*)} + \frac{\partial g}{\partial y} \bigg|_{(x^*,y^*)} \\
	&=&  \Bigg( 1 - \frac{2 x}{\gamma} - \frac{2 x y (1 + \alpha \xi)}{(1 + x^2 + \alpha \xi)^2} + \frac{\delta (x^2 + \xi)}{1+x^2+\alpha \xi} - m - 2 \epsilon y \Bigg)  \bigg|_{(x^*,y^*)}.
\end{eqnarray*}

Upon simplification, we have
\begin{equation}  \label{3iscdinttr}
	\begin{split}
		\text{Tr } J \bigg|_{(x^*,y^*)}  &=  \left( - \epsilon y - \frac{x}{\gamma} + \frac{xy (x^2 - 1 - \alpha \xi)}{(1+x^2 + \alpha \xi)^2} \right)  \bigg|_{(x^*,y^*)}. 
	\end{split}
\end{equation} 

The trace of the jacobian is negative when ${x^*}^2 < 1 + \alpha \xi$.

From (\ref{3iscdintdet}) and (\ref{3iscdinttr}), the following observations can be derived. 

\begin{itemize}
	\item In the absence of mutual interference (i.e., $\epsilon = 0$), the determinent (\ref{3iscdintdet}) is positive when $1 + \alpha \xi - \xi > 0$. The trace of the jacobian (\ref{3iscdinttr}) is negative when ${x^*}^2 < 1 + \alpha \xi$. Therefore, the sufficient condition for the stability of interior equilibrium is ${x^*}^2 < 1 + \alpha \xi$. This point is saddle when $1 + \alpha \xi - \xi < 0$. 
	\item When mutual interference is incorporated, the determinent of jacobian (\ref{3iscdintdet}) is positive when $1 + \alpha \xi - \xi > 0$ and ${x^*}^2 < 1 + \alpha \xi$. The trace of the jacobian (\ref{3iscdinttr}) is negative when ${x^*}^2 < 1 + \alpha \xi$. Therefore, the sufficient condition for the stability of interior equilibrium remains the same irrespective of the intra-specific competition. 
\end{itemize}

From (\ref{3iscdxstar}), it is observed that the system (\ref{3iscd}) can have atmost five real interior equilibria. Also, Lemma \ref{3iscdintcond} gives the condition on interior equilibrium as $\frac{\epsilon}{1+\frac{\epsilon}{\gamma}} < x^* < \gamma,\ \delta \xi - m (1+\alpha \xi) > - \gamma^2 (\delta - m)$. Incorporating this condition to the above obtained results, we have the following theorem dealing with the stability of the interior equilibrum point. 

\begin{thm} \label{inteqstab3iscd}
	The interior equilibrium of the system (\ref{3iscd}), if it exists, is
	\begin{enumerate} 
		\item[a. ] asymptotically stable when $1 + \alpha \xi - \xi > 0$ and $\frac{\epsilon}{1+\frac{\epsilon}{\gamma}} < x^* < \text{min } \{ \sqrt{1 + \alpha \xi}, \gamma\}$. 
		\item[b. ] saddle otherwise. 
	\end{enumerate}
\end{thm} 

\section{Bifurcation Analysis} \label{sec:3iscdbifur}

In this section, we study the various possible bifurcations exhibited by the system (\ref{3iscd}). We prove the transcritical and saddle-node bifurcations analytically. Further, we numerically simulate Hopf bifurcation, Focus-node transitions and Hysterisis loop in addition to these two bifurcations. 

\subsection{Transcritical Bifurcation}
Within this subsection, we derive the conditions for the existence of transcritical bifurcation near the equilibrium point $E_1 = (\gamma,0)$ using the parameter $\xi$ as the bifurcation parameter.

\begin{thm}
	When the parameter satisfies $1+ (1-\alpha)\gamma^2 \neq 0,\ \delta - m \alpha \neq 0$ and $\xi = \xi^* = \frac{m (1 + \gamma^2)-\delta \gamma^2}{\delta - m \alpha}$, a transcritical bifurcation occurs at $E_1 = (\gamma,0)$ in the system (\ref{3iscd}).
\end{thm}

\begin{proof}
	The jacobian matrix corresponding to equilibrium point $E_1=(\gamma,0)$ is given by
	
	\begin{equation*}
		J(E_1) = 
		\begin{bmatrix}
			-1  & \frac{-\gamma^2}{1 + \alpha \xi + \gamma^2} \\\
			0 & \frac{(\delta - m) \gamma^2 + \delta \xi - m (1+\alpha \xi)}{1 + \alpha \xi + \gamma^2}
		\end{bmatrix}.
	\end{equation*}
	The eigenvectors corresponding to the zero eigenvalues of $J(E_1)$ and $J(E_1)^{T}$ be denoted by $V$ and $W$, respectively. 
	
	\begin{equation*}
		V = \begin{bmatrix} V_1  \\ V_2 \end{bmatrix} = \begin{bmatrix}	1  \\ - \frac{1 + \alpha \xi + \gamma^2 }{\gamma^2} \end{bmatrix} , \ \ W = \begin{bmatrix} 0 \\ 1 \end{bmatrix}.
	\end{equation*}
	Note that $V_2 < 0$. Let us denote system (\ref{3iscd}) as $H = \begin{bmatrix} F  \\ G \end{bmatrix}.$
	Thus, 
	
	\begin{eqnarray*}
		H_{\xi} (E_1; \xi^*) &=&  \begin{bmatrix} 0 \\ 0 \end{bmatrix}, \\
		DH_\xi(E_1; \xi^*)V &=& 
		\begin{bmatrix}
			\frac{\partial F_\xi}{\partial x} & \frac{\partial F_\xi}{\partial y} \\
			\frac{\partial G_\xi}{\partial x} & \frac{\partial G_\xi}{\partial y}
		\end{bmatrix}
		\begin{bmatrix}
			V_1 \\
			V_2
		\end{bmatrix}
		_{(E_1; \xi^*)} \\
		&=& \frac{-1}{1+\alpha \xi + \gamma^2}
		\begin{bmatrix}
			\alpha \\
			\frac{\delta}{\gamma^2} \left(1 + (1-\alpha)\gamma^2 \right)
		\end{bmatrix}, \\ 
		D^2 H(E_1; \xi^*)(V, V) &=&
		\begin{bmatrix}
			\frac{\partial^2 F}{\partial x^2} V_1^2 + 2 \frac{\partial^2 F}{\partial x \partial y} V_1 V_2 + \frac{\partial^2 F}{\partial y^2} V_2^2 \\
			\frac{\partial^2 G}{\partial x^2} V_1^2 + 2 \frac{\partial^2 G}{\partial x \partial y} V_1 V_2 + \frac{\partial^2 G}{\partial y^2} V_2^2
		\end{bmatrix}
		_{(E_1, \xi^*)} \\
		&=& \begin{bmatrix}
			\frac{2 (1 + \alpha \xi - \gamma^2)}{ \gamma (1 + \alpha \xi + \gamma^2)} \\
			\frac{-2 \epsilon (1 + \alpha \xi + \gamma^2)^3 - 4 \delta \gamma^3 (1+\alpha \xi - \xi)}{\gamma^4 (1 + \alpha \xi + \gamma^2)}
		\end{bmatrix}.
	\end{eqnarray*}
	
	If $1+ (1-\alpha)\gamma^2 \neq 0$, we have
	
	\begin{eqnarray*}
		W^{T} H_{\xi} (E_1; \xi^*)&=& 0, \\
		W^{T} [DH_\xi(E_1; \xi^*)V]&=& \frac{-\delta (1+(1-\alpha)\gamma^2)}{\gamma^2 (1 + \alpha \xi + \gamma^2)} \neq 0, \\
		W^{T} [D^2 H(E_1; \xi^*)(V, V) ]&=& -2 \epsilon (1 + \alpha \xi + \gamma^2)^3 - 4 \delta \gamma^3 (1+\alpha \xi - \xi) \\ & & \neq 0.
	\end{eqnarray*}
	By the Sotomayor's theorem \cite{perko2013differential}, the system (\ref{3iscd}) undergoes a transcritical bifurcation around $E_1$ at $\xi = \xi^*$.
\end{proof}

In \autoref{trans3iscd}, we observe the transcritical bifurcation where the stability of interior equilibrium ($E^* = (x^*,y^*)$) and the axial eqilibrium ($E_1 = (\gamma,0)$) are exchanged at the bifurcation point $\xi = 2.0$. We depict the equilibria and their stability for the following set of parameter values, $\gamma = 1.0,\ \alpha = 1.0,\ \epsilon = 0.5,\ \delta = 8.0,\ m=6.0$. 

\begin{figure}[!ht]
	\centering
	\includegraphics[width=\textwidth]{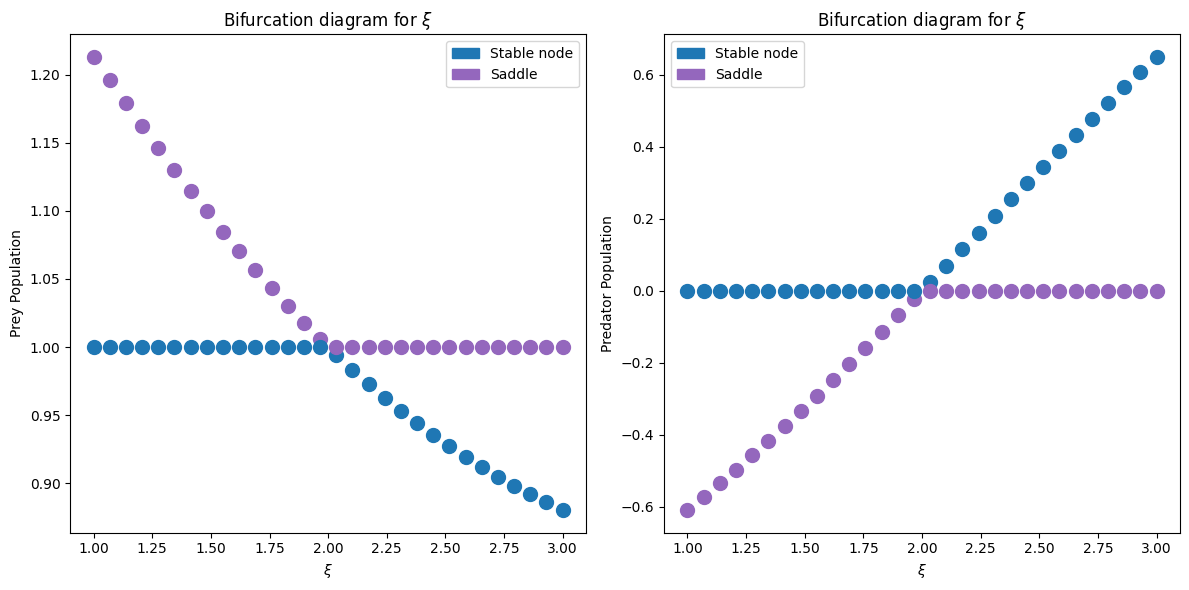}
	\caption{Transcritical bifurcation diagram around axial equilibrium $E_1 = (\gamma,0)$ with respect to the quantity of additional food $\xi$.}
	\label{trans3iscd}
\end{figure}

\subsection{Saddle-node Bifurcation}

Within this subsection, we derive the conditions for the existence of saddle-node bifurcation near the equilibrium point $E_2 = \left( 0, \frac{\delta \xi - m (1+\alpha \xi)}{\epsilon (1 + \alpha \xi)} \right)$ using the quantity of additional food ($\xi$) as the bifurcation parameter.

\begin{thm}
	When the parameter satisfies $\delta \neq m \alpha$ and $\xi = \xi^* = \frac{m}{\delta - m \alpha}$, a saddle-node bifurcation occurs at $E_2 = \left( 0, \frac{\delta \xi - m (1+\alpha \xi)}{\epsilon (1 + \alpha \xi)} \right)$ in the system (\ref{3iscd}).
\end{thm}

\begin{proof}
	The jacobian matrix corresponding to equilibrium point $E_2 = \left( 0, \frac{\delta \xi - m (1+\alpha \xi)}{\epsilon (1+\alpha \xi)} \right)$ is given by
	
	\begin{equation*}
		J(E_2) =
		\begin{bmatrix}
			1 & 0 \\
			0 & - \frac{\delta \xi - m (1 + \alpha \xi)}{1 + \alpha \xi}
		\end{bmatrix}.
	\end{equation*}
	
	The eigenvectors corresponding to the zero eigenvalues of $J(E_2)$ and $J(E_2)^{T}$ be denoted by $V$ and $W$, respectively. 
	
	\begin{equation*}
		V = \begin{bmatrix} V_1  \\ V_2 \end{bmatrix} = \begin{bmatrix}	0  \\ 1 \end{bmatrix}, \ \ W = \begin{bmatrix} 0  \\ 1 \end{bmatrix}.
	\end{equation*}
	
	Let us denote system (\ref{3iscd}) as $H = \left[ \begin{matrix} F  \\ G \end{matrix} \right].$
	Thus, 
	
	\begin{eqnarray*}
		H_{\xi} (E_2; \xi^*) &=&  \begin{bmatrix} 0 \\ \frac{\delta (\delta \xi - m (1+\alpha\xi))}{\epsilon (1 + \alpha \xi)^3} \end{bmatrix}, \\
		DH_\xi(E_2; \xi^*)V &=& 
		\begin{bmatrix}
			\frac{\partial F_\xi}{\partial x} & \frac{\partial F_\xi}{\partial y} \\
			\frac{\partial G_\xi}{\partial x} & \frac{\partial G_\xi}{\partial y}
		\end{bmatrix}
		\begin{bmatrix}
			V_1 \\
			V_2
		\end{bmatrix}
		_{(E_2; \xi^*)} =
		\begin{bmatrix}	0 \\
			\frac{\delta}{(1 + \alpha \xi)^2}
		\end{bmatrix}, \\
		D^2 H(E_2; \xi^*)(V, V) &=&
		\begin{bmatrix}
			\frac{\partial^2 F}{\partial x^2} V_1^2 + 2 \frac{\partial^2 F}{\partial x \partial y} V_1 V_2 + \frac{\partial^2 F}{\partial y^2} V_2^2 \\
			\frac{\partial^2 G}{\partial x^2} V_1^2 + 2 \frac{\partial^2 G}{\partial x \partial y} V_1 V_2 + \frac{\partial^2 G}{\partial y^2} V_2^2
		\end{bmatrix}_{(E_2, \xi^*)} = \begin{bmatrix} 0 \\ - 2 \epsilon	\end{bmatrix}.
	\end{eqnarray*}
	
	If $\delta \neq m \alpha$, we have
	
	\begin{eqnarray*}
		W^{T} H_{\xi} (E_2; \xi^*) &=&  \frac{\delta (\delta \xi - m (1+\alpha\xi))}{\epsilon (1 + \alpha \xi)^3}  \neq 0, \\
		W^{T} [DH_\xi(E_2; \xi^*)V] &=& \frac{\delta}{(1 + \alpha \xi)^2} \neq 0, \\
		W^{T} [D^2 H(E_2; \xi^*)(V, V) ] &=& - 2 \epsilon \neq 0.
	\end{eqnarray*}
	By the Sotomayor's theorem \cite{perko2013differential}, the system (\ref{3iscd}) undergoes a saddle-node bifurcation around $E_2$ at $\xi = \xi^*$.
\end{proof}

In \autoref{saddle3iscd}, the saddle-node bifurcation around the another axial equilibrium $E_2$ with respect to $\xi$ is discussed. For this same set of parameter values, both the equilibria $E^*$ and $E_2$ exist and move towards each other as $\xi$ reduces. Further at $\xi = 3.0$, both ($E^*$ and $E_2$) collide and after which only $E^*$ exists which ensures the happening of saddle-node bifurcation. When $\xi < 3.0$, then there does not exist any prey-free equilibrium $E_2$. The remaining parameters are as follows: $\gamma = 1.0,\ \alpha = 1.0,\ \epsilon = 0.5,\ \delta = 8.0,\ m=6.0$.

\begin{figure}[!ht]
	\centering
	\includegraphics[width=\textwidth]{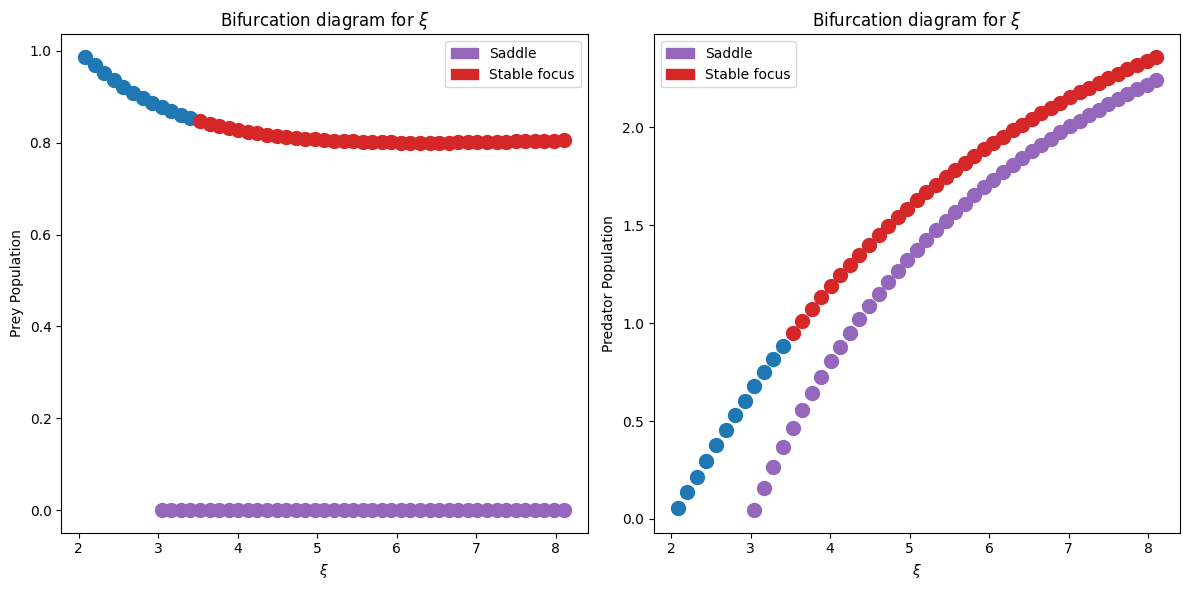}
	\caption{Saddle-node bifurcation diagram around axial equilibrium $E_2 = \left(0,\frac{\delta \xi - m (1 + \alpha \xi)}{\epsilon (1+\alpha \xi)}\right)$ with respect to the quantity of additional food $\xi$.}
	\label{saddle3iscd}
\end{figure}

\subsection{Hopf Bifurcation}

The existence of Hopf bifurcation with respect to $\epsilon$ is depicted in \autoref{hopf3iscd}. Stable limit cycle is observed around the interior equilibrium when $\epsilon = 0.035$ and the limit cycle disappears when $\epsilon$ is increased to $0.045$. The remaining parameter values are as follows: $\gamma = 15.0,\ \alpha = 0.1,\ \xi = 0.45,\ \delta = 0.45,\ m=0.28$.

\begin{figure}[!ht]
	\centering
	\includegraphics[width=\textwidth]{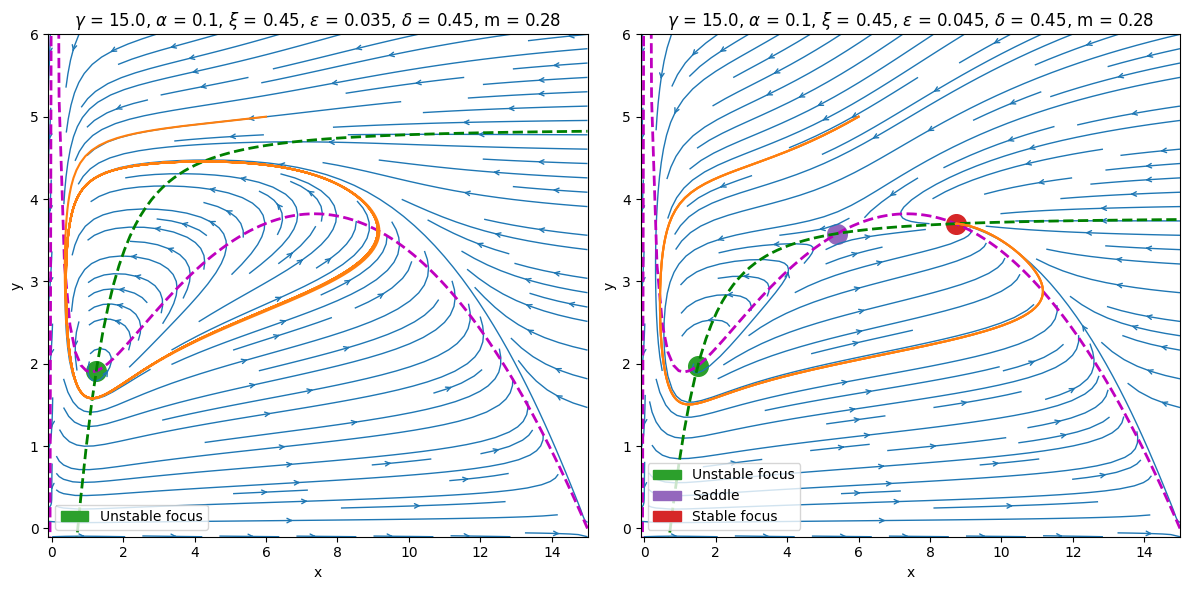}
	\caption{Supercritical Hopf bifurcation diagram with respect to the intra-specific competition $\epsilon$.}
	\label{hopf3iscd}
\end{figure}

\subsection{Hysteresis Loop}

The interior equilibria corresponding to different values of $\epsilon$ are illustrated in \autoref{SPlot3iscd}. The subplots display the intra-specific population plotted against the interior prey and predator equilibrium values, respectively. The parameter values used for this numerical simulation are $
\gamma = 15.0, \ \alpha = 0.1, \ \xi = 1.0,\  \delta = 0.3,\ m= 0.258$. This figure demonstrates the presence of multiple interior equilibria across the parameter range $\epsilon \in [0.002, 0.02]$, highlighting the system’s nonlinear response and the emergence of potential bistable regimes within this interval.

This S-shaped plot also depicts the existence of two saddle-node bifurcations at the two-turning points (folds) of the S-curves in both the panels. Also two focus-node transitions are observed in these plots. Though eigenvalue types are changing from real to complex, these are not bifurcations in the topological sense. In addition to this, a change in stable focus to unstable focus is observed in this plot. However, this doesn't guarentee the existence or disappearance of limit cycle. Therefore, this stability transition of a fixed point need not be a Hopf bifurcation. 

\begin{figure}[!ht]
	\centering
	\includegraphics[width=\textwidth]{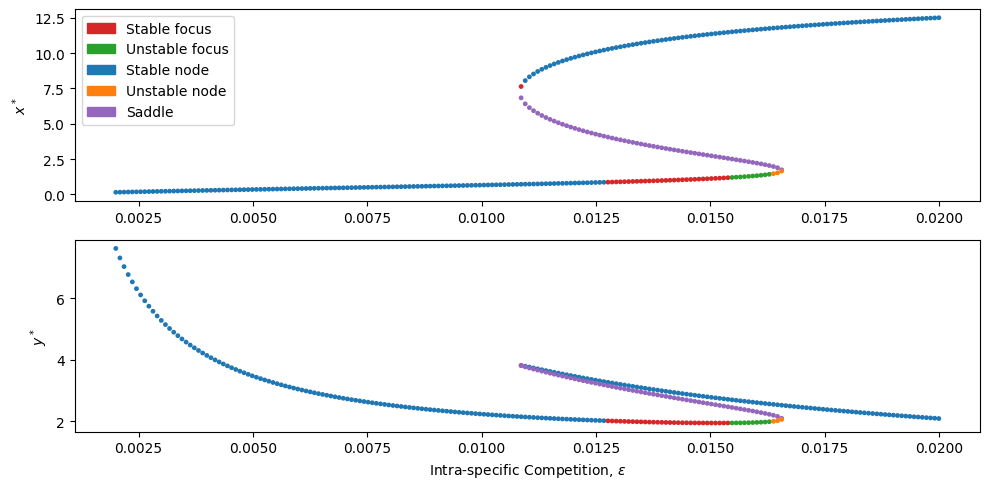}
	\caption{Nature of interior equilibria with respect to the intra-specific competition ($\epsilon$)}
	\label{SPlot3iscd}
\end{figure}

In figure \ref{SPlot3iscd}, We observe that the system also exhibits two stable interior equilibria In the region between the two saddle-node bifurcations. This bistability behavior gives rise to a phenomenon called Hysteresis. When multiple stable equilibria coexists for a range of parameters (here, $\epsilon$), the situation where the state of the system depends on the history of the parameter is referred to as Hysteresis. 

\autoref{SPlot3iscd} presents two stable equilibria branches. Let us refer the bottom and top stable branches in this S-shaped curve as the branch 1 and 2 respectively. As $\epsilon$ increases, stability remains in the branch 1 until the second saddle-node bifurcation point is reached. Then there will be a jump in the stability to the branch 2. Similarly, as $\epsilon$ reduces, the system continues to remain stable in branch 2 until it jumps back at the other fold. Therefore, the history of $\epsilon$ determines the stable equilibria in the bistable behavior. 

Now, when we consider a time-varying $\epsilon$ in the region, we observe the hysteresis loop. In figure \ref{Hysterisis3iscd}, we consider the $\epsilon$ as smooth periodic sweep using sine in the region $(0.002, 0.02)$. \autoref{Hysterisis3iscd} represents the hysteresis loop present in the plots with prey interior equilibria, predator interior equilibria, phase portrait respectively. They go up following one branch and come down following different branch. They intersect at the two saddle-node bifurcation points resulting in a loop known as Hysteresis loop.

\begin{figure}[!ht]
	\centering
	\includegraphics[width=0.8\textwidth]{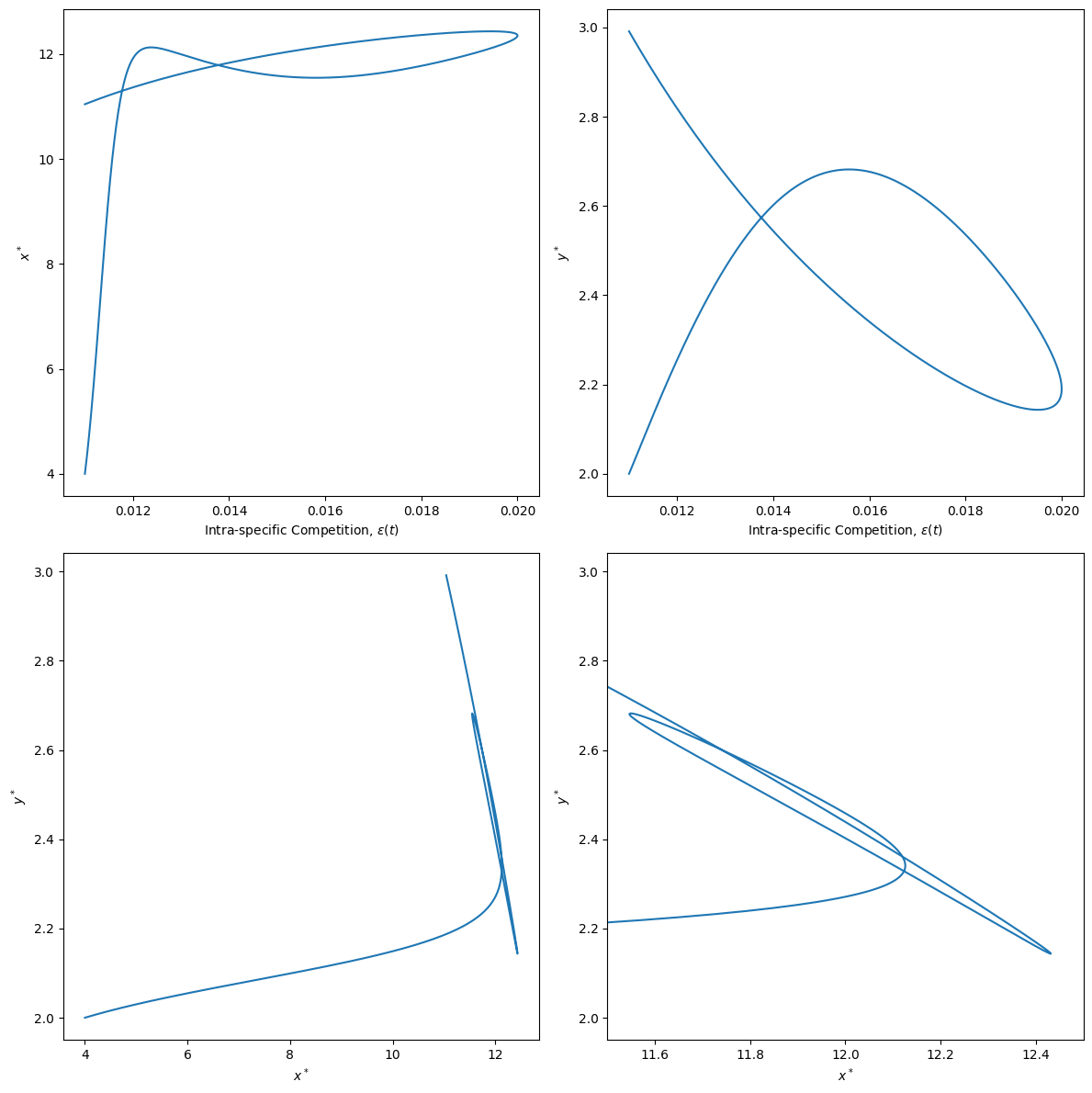}
	\caption{Hysterisis loop exhibited by the system (\ref{3iscd}) with respect to the time-dependent intra-specific competition $\epsilon (t)$.}
	\label{Hysterisis3iscd}
\end{figure}

\section{Global Dynamics}
\label{sec:3iscdglobaldynamics}

In this section, we study the global dynamics of the system (\ref{3iscd}) in the $\alpha$-$\xi$ parameter space. For this study, we divide the parameter space of the system (\ref{3iscd}) in the absence of additional food into three regions. These regions are divided based on the qualitative behaviors of the interior equilibria of the system. They are divided into three regions, namely, $R_1,\ R_2,\ R_3$ corresponding to the space where there is no interior equilibria, stable node and the stable focus respectively. \autoref{initial3iscd} depicts the phase portrait of the initial system in all three regions. 

\begin{figure}[!ht]
	\centering
	\includegraphics[width=\linewidth]{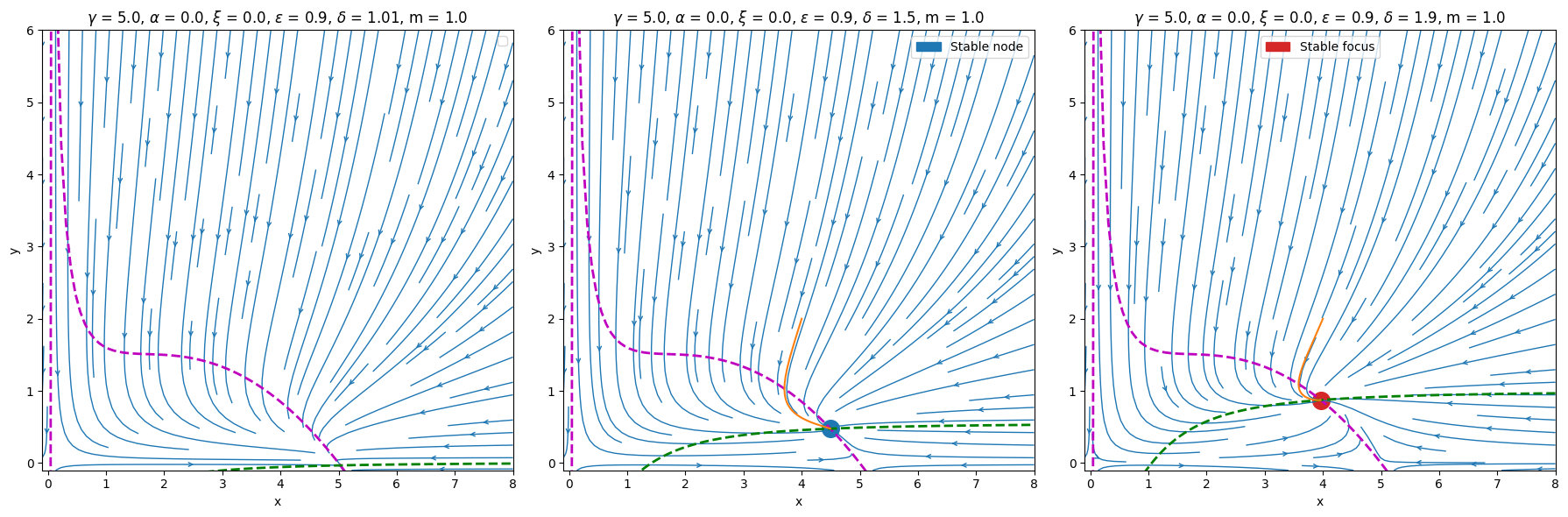}  
	\caption{Dynamics of the system (\ref{3iscd}) in the absence of additional food.}
	\label{initial3iscd}
\end{figure}

Now, in each of these regions, we study the influence of additional food by dividing the $\alpha - \xi$ parameter space into regions based on the following curves. Each of these curves divide the space into two regions based on the qualitative nature of the equilibrium point. 

\begin{equation*}
	\begin{split}
		\text{Bifurcation Curve for } E_0: &\ \phi_1(\alpha,\xi) : \  \delta \xi - m (1 + \alpha \xi) = 0. \\
		\text{Bifurcation Curve for } E_1: &\ \phi_2(\alpha,\xi) : \ \delta \xi - m (1 + \alpha \xi) + (\delta - m) \gamma^2 = 0. \\
		\text{Existence Curve for } E^*: &\ \phi_3(\alpha,\xi) : \ 1 + \alpha \xi - \xi = 0. \\			
	\end{split}
\end{equation*} 

\begin{figure}[!ht]
	\centering
	\includegraphics[width=0.7\linewidth]{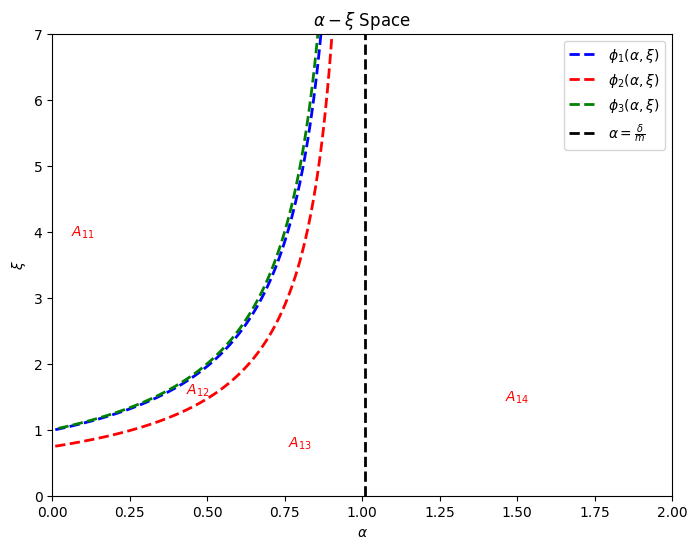} \\
	\caption{Influence of additional food on the system (\ref{3iscd}) when the parameters belong to the region $R_1$.}
	\label{r13iscd}
\end{figure}

\autoref{r13iscd} divides the $\alpha - \xi$ space into $4$ regions. In this region, there is no interior equilibrium in the absence of additional food. As additional food is provided, $E_0$ is unstable and the rest are saddle in nature in the region $A_{11}$. When we move to the region $A_{12}$, $E_2$ ceases to exist and $E^*$ will be the only stable equilibrium. By further increasing the additional food quality and quantity, we will get a bistability condition in the regions $A_{13}$ and $A_{14}$ as $E_1$ and $E^*$ are stable. However, provision of additional food in the regions $A_{13}$ and $A_{14}$ takes us to the original qualitative behavior as in the case of absence of additional food based on the initial prey and predator populations.

\begin{figure}[!ht]
	\centering
	\includegraphics[width=0.7\linewidth]{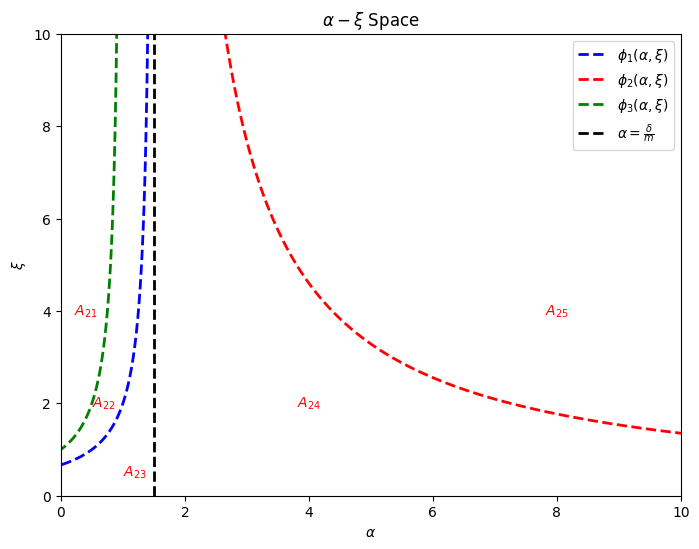} \\
	\caption{Influence of additional food on the system (\ref{3iscd}) when the parameters belong to the region $R_2$.}
	\label{r23iscd}
\end{figure}

\begin{figure}[!ht]
	\centering
	\includegraphics[width=0.7\linewidth]{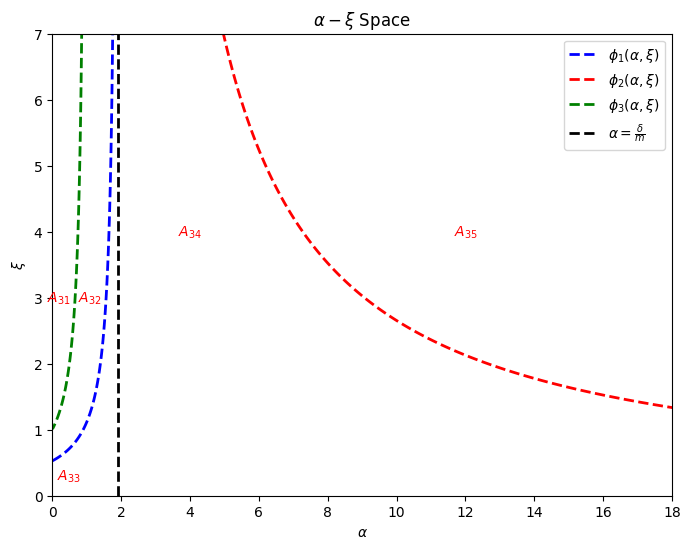} \\
	\caption{Influence of additional food on the system (\ref{3iscd}) when the parameters belong to the region $R_3$.}
	\label{r33iscd}
\end{figure}

\autoref{r23iscd}  and \autoref{r33iscd} divides the $\alpha - \xi$ space into $5$ regions each. In these regions, there is a stable interior equilibrium in the absence of additional food. As additional food is provided, $E_1,\ E_2,\ E^*$ are saddle in the regions $A_{21}$ and $A_{31}$. The interior equilibrium $E^*$ and the predator-free equilibrium $E_1$ are stable in regions $A_{25}$ and $A_{35}$ resulting in bistability of the system. In rest of the regions, only interior equilibrium is stable. Therefore, as additional food is provided the system continues to exhibit the original nature of initial system depending on the initial prey and predator populations.

\section{Consequences of providing Additional Food} \label{sec:3iscdconseq}

\indent In this section, we present the consequences of providing additional food by studying the possibility of existence in three different scenarios: Pest eradication, pest dominance and the coexistence of pest and natural enemies. 

As pest-free equilibrium $E_2$ is saddle whenever it exists, the arbitrary provision of additional food can never lead to the complete pest eradication for this system irrespective of the initial behaviour of the system in the absence of additional food. 

Now consider the case of pest dominance. In this system, there is no clear case of pest dominance possible as the natural enemy free equilibrium is stable only in the regions $A_{13}, \ A_{14},\ A_{25}$ and $A_{35}$. However in these regions, the system exhibits bistability where both the equilibria $E^*$ and $E_1$ are stable. The dynamics goes to one of these equilibria depending on the initial condition. Therefore, altering the natural enemy population can change the stability of the system and avoid pest dominance stage. 

In the rest of the regions, only interior equilibria $E^*$ is stable. As the system can exhibit upto 5 interior equilibria and the real positive equilibria are either saddle or stable, the system can exhibit bistability among the interior equilibria which highlights the importance of initial prey and predator populations. Therefore, choosing the initial populations so that the system lead to the interior equilibria with least pest population will be the best strategy to reduce pest dominance. Also it is important to note that the minimum pest population that can be reached is $\frac{\epsilon}{1+\frac{\epsilon}{\gamma}}$. This highlights the importance of intra-specific competition for the pest management.

%\autoref{equiplot4iscd} numerically depicts the stability nature of various equilibria that the system exhibits only when additional food and competition terms are altered. This shows the importance of these two terms in the dynamics of the system (\ref{4iscd}). Each frame depicts the existence of $0-3$ interior equilibrium.
%
%\begin{figure}[ht]
%	\includegraphics[width=\textwidth]{equiplot4iscd.png}
%	\caption{The stability nature of various equilibria of the system (\ref{4iscd}).}
%	\label{equiplot4iscd}
%\end{figure}

\section{Time-Optimal Control Studies for Holling type-III Systems} \label{sec:3iscdtimecontrol}

In this section, we formulate and characterise two time-optimal control problems with quality of additional food and quantity of additional food as control parameters respectively. We shall drive the system (\ref{3iscd}) from the initial state $(x_0,y_0)$ to the final state $(\bar{x},\bar{y})$ in minimum time.

\subsection{Quality of Additional Food as Control Parameter}

We assume that the quantity of additional food $(\xi)$ is constant and the quality of additional food varies in $[\alpha_{\text{min}},\alpha_{\text{max}}]$. The time-optimal control problem with additional food provided prey-predator system involving Holling type-III functional response and intra-specific competition among predators (\ref{3iscd}) with quality of additional food ($\alpha$) as control parameter is given by

\begin{equation}
	\begin{rcases}
		& \displaystyle {\bf{\min_{\alpha_{\min} \leq \alpha(t) \leq \alpha_{\max}} T}} \\
		& \text{subject to:} \\
		& \dot{x}(t) = x(t) \left( 1 - \frac{x(t)}{\gamma} \right) - \frac{x^2(t) y(t)}{1 + x^2(t) + \alpha(t) \xi}, \\
		& \dot{y}(t) = \delta y(t) \left( \frac{x^2(t) + \xi}{1+x^2(t)+\alpha(t) \xi} \right) - m y(t) - \epsilon y^2(t), \\
		& (x(0),y(0)) = (x_0,y_0) \ \text{and} \ (x(T),y(T)) = (\bar{x},\bar{y}).
	\end{rcases}
	\label{3iscdalpha0}
\end{equation}

This problem can be solved using a transformation on the independent variable $t$ by introducing an independent variable $s$ such that $\mathrm{d}t = (1 + \alpha \xi + x^2) \mathrm{d}s$. This transformation converts the time-optimal control problem ($\ref{3iscdalpha0}$) into the following linear problem.

\begin{equation}
	\begin{rcases}
		& \displaystyle {\bf{\min_{\alpha_{\min} \leq \alpha(t) \leq \alpha_{\max}} S}} \\
		& \text{subject to:} \\
		& \dot{x}(s) = x \left(1 - \frac{x}{\gamma}\right) \left(1 + x^2 + \alpha \xi \right) - x^2 y, \\
		& \dot{y}(s) = \delta (x^2 + \xi) y - (1 + x^2 + \alpha \xi) (m y + \epsilon y^2), \\
		& (x(0),y(0)) = (x_0,y_0) \ \text{and} \ (x(S),y(S)) = (\bar{x},\bar{y}).
	\end{rcases}
	\label{3iscdalpha}
\end{equation}

Hamiltonian function for this problem (\ref{3iscdalpha}) is given by
\begin{equation*}
	\begin{split}
		\mathbb{H}(s,x,y,p,q) =& p \left[x \left(1 - \frac{x}{\gamma}\right) \left(1 + x^2 + \alpha \xi \right) - x^2 y\right] \\
		& + q \left[\delta (x^2 + \xi) y - (1 + x^2 + \alpha \xi) (m y + \epsilon y^2) \right] \\
		=& \left[ p x \xi \left(1 - \frac{x}{\gamma}\right) - q \xi \left(m y + \epsilon y^2 \right) \right] \alpha \\ 
		&+ \left[ p x \left(\left(1 - \frac{x}{\gamma}\right) (1 + x^2) - xy\right) + qy \left( \delta (x^2 + \xi)- (1 + x^2) (m+ \epsilon y)\right)\right]. \\
	\end{split}
\end{equation*}

Here, $p$ and $q$ are costate variables satisfying the adjoint equations 
\begin{equation*}
	\begin{split}
		\dot{p} =& -p \left[ 1+\alpha \xi  - \frac{2 (1+\alpha \xi)}{\gamma} x - 2 x y + 3 x^2 - \frac{4}{\gamma} x^3\right] -  2 q x y \left(\delta - m - \epsilon y \right), \\
		\dot{q} =& p x^2 - q \left[ \delta (x^2 + \xi) - (1 + x^2 + \alpha \xi) (m + 2 \epsilon y) \right].
	\end{split}
\end{equation*}

Since Hamiltonian is a linear function in $\alpha$, the optimal control can be a combination of bang-bang and singular controls \cite{cesari2012optimization}. Since we are minimizing the Hamiltonian, the optimal strategy is given by 

\begin{equation}
	\alpha^*(t) =
	\begin{cases}
		\alpha_{\max}, &\text{ if } \frac{\partial \mathbb{H}}{\partial \alpha} < 0\\
		\alpha_{\min}, &\text{ if } \frac{\partial \mathbb{H}}{\partial \alpha} > 0
	\end{cases}
\end{equation}
where
\begin{equation}
	\frac{\partial \mathbb{H}}{\partial \alpha} = p x \xi \left(1 - \frac{x}{\gamma}\right) - q \xi \left(m y + \epsilon y^2\right).
\end{equation}

This problem \ref{3iscdalpha} admits a singular solution if there exists an interval $[s_1,s_2]$ on which $\frac{\partial \mathbb{H}}{\partial \alpha} = 0$. Therefore, 

\begin{equation}
	\frac{\partial \mathbb{H}}{\partial \alpha} = p x \xi \left(1 - \frac{x}{\gamma}\right) - q \xi \left(m y + \epsilon y^2\right) = 0. \textit{ i.e. } \frac{p}{q} = \frac{ m y + \epsilon y^2}{x \left(1 - \frac{x}{\gamma}\right)}. \label{3iscdapbyq1}
\end{equation}

Differentiating $\frac{\partial \mathbb{H}}{\partial \alpha}$ with respect to $s$ we obtain 
\begin{equation*}
	\begin{split}    
		\frac{\mathrm{d}}{\mathrm{d}s} \frac{\partial \mathbb{H}}{\partial \alpha} =&\frac{\mathrm{d}}{\mathrm{d}s} \left[ p x \xi \left(1 - \frac{x}{\gamma}\right) - q \xi \left(m y + \epsilon y^2\right) \right] \\
		=& \xi x \left(1-\frac{x}{\gamma}\right) \dot{p} + p \xi \left(1 - \frac{2 x}{\gamma}\right) \dot{x} - \xi (m y + \epsilon y^2) \dot{q} - q \xi (m +2 \epsilon y) \dot{y}.
	\end{split}
\end{equation*}

Substituting the values of $\dot{x}, \dot{y}, \dot{p}, \dot{q}$ in the above equation, we obtain
\begin{equation*}
	\begin{split}
		\frac{\mathrm{d}}{\mathrm{d}s} \frac{\partial \mathbb{H}}{\partial \alpha} =&  \xi x \left(1-\frac{x}{\gamma}\right) \Bigg(-p \left[ 1+\alpha \xi  - \frac{2 (1+\alpha \xi)}{\gamma} x - 2 x y + 3 x^2 - \frac{4}{\gamma} x^3\right] \\
		& -  2 q x y \left(\delta - m - \epsilon y \right) \Bigg) \\
		&+ p \xi \left(1 - \frac{2 x}{\gamma}\right) \left( x \left(1 - \frac{x}{\gamma}\right) \left(1 + x^2 + \alpha \xi \right) - x^2 y \right) \\
		& - \xi (m y + \epsilon y^2) \left( p x^2 - q \left[ \delta (x^2 + \xi) - (1 + x^2 + \alpha \xi) (m + 2 \epsilon y) \right] \right) \\
		& - q \xi (m +2 \epsilon y) \left( \delta (x^2 + \xi) y - (1 + x^2 + \alpha \xi) (m y + \epsilon y^2)\right).
	\end{split}
\end{equation*}

Along the singular arc, $\frac{\mathrm{d}}{\mathrm{d}s} \frac{\partial \mathbb{H}}{\partial \alpha} = 0$. This implies that 
\begin{equation*}
	\begin{split}
		p \xi x \left[ \frac{-(1+\alpha \xi)}{\gamma} + \frac{1+\alpha \xi}{\gamma} x + x y - 2 x^2 + \frac{4}{\gamma} x^3 - \frac{2}{\gamma^2} x^4 - x y \left(m + \epsilon y\right)\right] & \\
		-q \xi y \left[ 2 x^2 \left(1 - \frac{x}{\gamma} \right) \left(\delta - m - 2 \epsilon \right) + \delta \epsilon y (x^2 + \xi)\right]   & = 0.
	\end{split}
\end{equation*}

and that 
\begin{equation} \label{3iscdapbyq2}
	\frac{p}{q} = \frac{y}{x} \ \  \frac{2 x^2 \left(1 - \frac{x}{\gamma} \right) \left(\delta - m - 2 \epsilon \right) + \delta \epsilon y (x^2 + \xi)}{\frac{-(1+\alpha \xi)}{\gamma} + \frac{1+\alpha \xi}{\gamma} x + x y - 2 x^2 + \frac{4}{\gamma} x^3 - \frac{2}{\gamma^2} x^4 - x y \left(m + \epsilon y\right)}.
\end{equation}

The solutions of the system of equations (\ref{3iscdapbyq1}) and (\ref{3iscdapbyq2}) gives the switching points of the bang-bang control.

\subsection{Quantity of Additional Food as Control Parameter}

In this section, We assume that the quality of additional food $(\alpha)$ is constant and the quantity of additional food varies in $[\xi_{\text{min}},\xi_{\text{max}}]$. The time-optimal control problem with additional food provided prey-predator system involving Holling type-III functional response and intra-specific competition among predators (\ref{3iscd}) with quantity of additional food ($\xi$) as control parameter is given by

\begin{equation}
	\begin{rcases}
		& \displaystyle {\bf{\min_{\xi_{\min} \leq \xi(t) \leq \xi_{\max}} T}} \\
		& \text{subject to:} \\
		& \dot{x}(t) = x(t) \left( 1 - \frac{x(t)}{\gamma} \right) - \frac{x^2(t) y(t)}{1 + x^2(t) + \alpha \xi(t)}, \\
		& \dot{y}(t) = \delta y(t) \left( \frac{x^2(t) + \xi(t)}{1+x^2(t)+\alpha \xi(t)} \right) - m y(t) - \epsilon y^2(t), \\
		& (x(0),y(0)) = (x_0,y_0) \ \text{and} \ (x(T),y(T)) = (\bar{x},\bar{y}).
	\end{rcases}
	\label{3iscdxi0}
\end{equation}

This problem can be solved using a transformation on the independent variable $t$ by introducing another independent variable $s$ such that $\mathrm{d}t = (1 + \alpha \xi + x^2) \mathrm{d}s$. This transformation converts the time-optimal control problem ($\ref{3iscdxi0}$) into the following linear problem.

\begin{equation}
	\begin{rcases}
		& \displaystyle {\bf{\min_{\xi_{\min} \leq \xi(t) \leq \xi_{\max}} S}} \\
		& \text{subject to:} \\
		& \dot{x}(s) = x \left(1 - \frac{x}{\gamma}\right) (1 + x^2 + \alpha \xi) - x^2 y, \\
		& \dot{y}(s) = \delta (x^2 + \xi) y - (1 + x^2 + \alpha \xi) (m y + \epsilon y^2), \\
		& (x(0),y(0)) = (x_0,y_0) \ \text{and} \ (x(S),y(S)) = (\bar{x},\bar{y}).
	\end{rcases}
	%\text{$\alpha$ - Control}
	\label{3iscdxi}
\end{equation}

Hamiltonian function for this problem (\ref{3iscdxi}) is given by
\begin{equation*}
	\begin{split}
		\mathbb{H}(s,x,y,p,q) =& p \left[ x \left(1 - \frac{x}{\gamma}\right) (1 + x^2 + \alpha \xi) - x^2 y \right] \\
		&+ q \left[\delta (x^2 + \xi) y - (1 + x^2 + \alpha \xi) (m y + \epsilon y^2) \right] \\
		=& \left[ p x \left(1 - \frac{x}{\gamma}\right) \alpha + \delta q y- q \alpha y (m + \epsilon y) \right] \xi \\ 
		& + \left[ p x \left(\left(1 - \frac{x}{\gamma}\right) (1 + x^2) - xy \right) + q y \left(\delta x^2 - (1 + x^2) (m+\epsilon y)\right)\right]. \\
	\end{split}
\end{equation*}

Here, $p$ and $q$ are costate variables satisfying the adjoint equations 
\begin{equation*}
	\begin{split}
		\dot{p} =& -p \left[ 1+\alpha \xi  - \frac{2 (1+\alpha \xi)}{\gamma} x - 2 x y + 3 x^2 - \frac{4}{\gamma} x^3\right] -  2 q x y \left(\delta - m - \epsilon y \right), \\
		\dot{q} =& p x^2 - q \left[ \delta (x^2 + \xi) - (1 + x^2 + \alpha \xi) (m + 2 \epsilon y) \right].
	\end{split}
\end{equation*}

Since Hamiltonian is a linear function in $\xi$, the optimal control can be a combination of bang-bang and singular controls \cite{cesari2012optimization}. Since we are minimizing the Hamiltonian, the optimal strategy is given by 

\begin{equation}
	\xi^*(t) =
	\begin{cases}
		\xi_{\max}, &\text{ if } \frac{\partial \mathbb{H}}{\partial \xi} < 0. \\
		\xi_{\min}, &\text{ if } \frac{\partial \mathbb{H}}{\partial \xi} > 0.
	\end{cases}
\end{equation}
where
\begin{equation}
	\frac{\partial \mathbb{H}}{\partial \xi} = p x \left(1 - \frac{x}{\gamma}\right) \alpha + \delta q y- q \alpha y (m + \epsilon y). 
\end{equation}

This problem (\ref{3iscdxi}) admits a singular solution if there exists an interval $[s_1,s_2]$ on which $\frac{\partial \mathbb{H}}{\partial \xi} = 0$. Therefore, 
\begin{equation}
	\frac{\partial \mathbb{H}}{\partial \xi} = p x \left(1 - \frac{x}{\gamma}\right) \alpha + \delta q y- q \alpha y (m + \epsilon y) = 0. \textit{ i.e. } \frac{p}{q} = \frac{ y \left( \alpha (m + \epsilon y)- \delta\right)}{\alpha x \left(1 - \frac{x}{\gamma}\right)}. \label{3iscdxpbyq1}
\end{equation}

Differentiating $\frac{\partial \mathbb{H}}{\partial \xi}$ with respect to $s$ we obtain 

\begin{equation*}
	\begin{split}
		\frac{\mathrm{d}}{\mathrm{d}s} \frac{\partial \mathbb{H}}{\partial \xi} &= \frac{\mathrm{d}}{\mathrm{d}s} \left[ p x \left(1 - \frac{x}{\gamma}\right) \alpha + \delta q y- q \alpha y (m + \epsilon y) \right] \\
		&= \alpha p \left(1-\frac{2 x}{\gamma} \right) \dot{x} + q (\delta - m \alpha - 2 \epsilon \alpha y) \dot{y} + \alpha x \left(1-\frac{x}{\gamma} \right) \dot{p} + \left(\delta-\alpha (m+\epsilon y)\right) y \dot{q}.
	\end{split}
\end{equation*}

Substituting the values of $\dot{x}, \dot{y}, \dot{p}, \dot{q}$ in the above equation, we obtain
\begin{equation*}
	\begin{split}
		\frac{\mathrm{d}}{\mathrm{d}s} \frac{\partial \mathbb{H}}{\partial \xi} =& \alpha p \left(1-\frac{2 x}{\gamma} \right) \left( x \left(1 - \frac{x}{\gamma}\right) (1 + x^2 + \alpha \xi) - x^2 y \right) \\
		&+ q (\delta - m \alpha - 2 \epsilon \alpha y) \left( \delta (x^2 + \xi) y - (1 + x^2 + \alpha \xi) (m y + \epsilon y^2) \right) \\
		&+ \alpha x \left(1-\frac{x}{\gamma} \right) \Bigg( -p \left[ 1+\alpha \xi  - \frac{2 (1+\alpha \xi)}{\gamma} x - 2 x y + 3 x^2 - \frac{4}{\gamma} x^3\right] \\
		& -  2 q x y \left(\delta - m - \epsilon y \right) \Bigg) \\
		&+ \left(\delta-\alpha (m+\epsilon y)\right) y \left( p x^2 - q \left[ \delta (x^2 + \xi) - (1 + x^2 + \alpha \xi) (m + 2 \epsilon y) \right]\right).
	\end{split}
\end{equation*}

Along the singular arc, $\frac{\mathrm{d}}{\mathrm{d}s} \frac{\partial \mathbb{H}}{\partial \xi} = 0$. This implies that 

\begin{equation*}
	\begin{split}
		p x^2 \left(- 2 \alpha x \left(1 - \frac{x}{\gamma}\right)^2 +  \left( \alpha + \delta-\alpha (m+\epsilon y)\right) y \right) & \\
		+ q y \left( \delta \epsilon y \left( 1 + x^2 - \alpha x^2 \right)- 2 \alpha x^2 \left(\delta - m - \epsilon y \right) \left(1-\frac{x}{\gamma} \right)  \right) & = 0.
	\end{split}
\end{equation*}

and that 
\begin{equation} \label{3iscdxpbyq2}
	\frac{p}{q} = \frac{y}{x^2} \ \ \frac{\delta \epsilon y \left( 1 + x^2 - \alpha x^2 \right)- 2 \alpha x^2 \left(\delta - m - \epsilon y \right) \left(1-\frac{x}{\gamma} \right)}{2 \alpha x \left(1 - \frac{x}{\gamma}\right)^2 -  \left( \alpha + \delta-\alpha (m+\epsilon y)\right) y}.
\end{equation}

The solutions of the system of equations (\ref{3iscdxpbyq1}) and (\ref{3iscdxpbyq2}) gives the switching points of the bang-bang control.

\subsection{Applications to Pest Management}

In this subsection, we simulated the time-optimal control problems (\ref{3iscdalpha}) and (\ref{3iscdxi}) using CasADi in python \cite{CasADi}. We implemented the direct transcription method with multiple shooting in order to solve the time-optimal control problems. In this method, we discretize the control problem into smaller intervals using finite difference integration, specifically the fourth-order Runge-Kutta (RK4) method. By breaking the trajectory into multiple shooting intervals, the state and control variables at each node are treated as optimization variables. The dynamics of the system are enforced as constraints between nodes, allowing for greater flexibility and improved convergence when solving the nonlinear programming problem with CasADi’s solvers.

\begin{figure}[!ht]
	\centering
	\includegraphics[width=\textwidth]{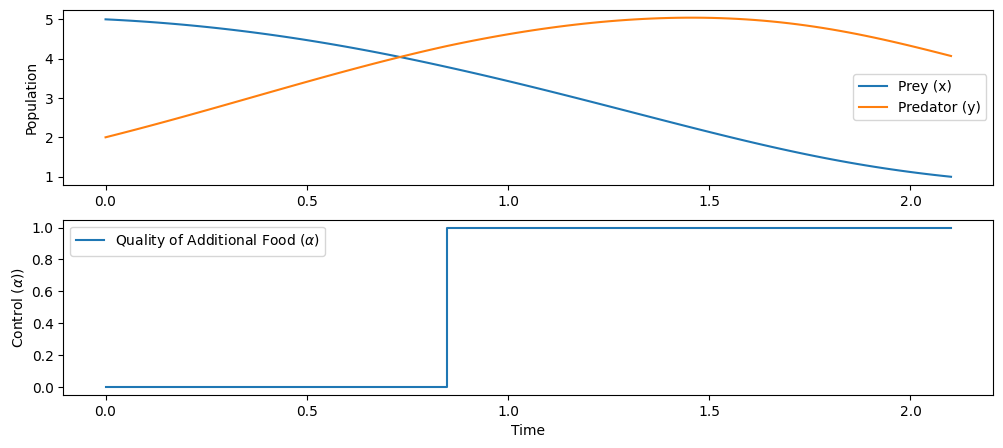}
	\caption{The optimal state trajectories and the optimal control trajectories for the time optimal control problem (\ref{3iscdalpha}).}
	\label{c3iscdalpha}
\end{figure}

\begin{figure}[!ht]
	\centering
	\includegraphics[width=\textwidth]{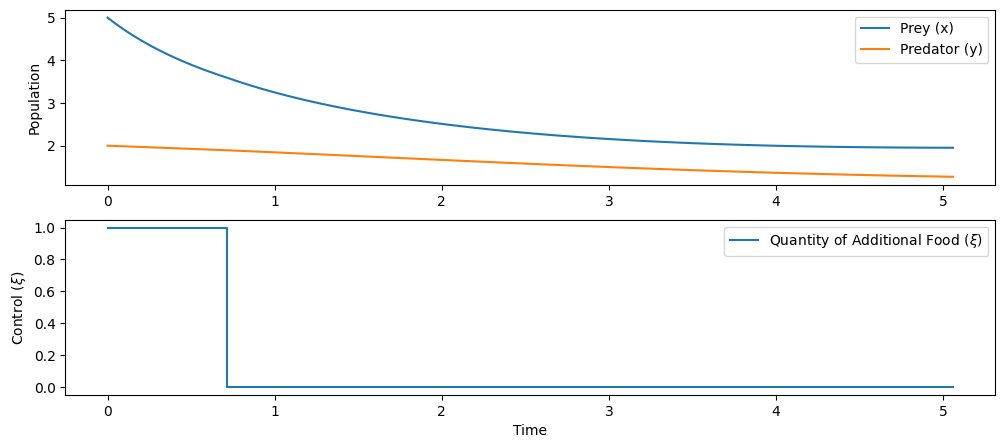}
	\caption{The optimal state trajectories and the optimal control trajectories for the time optimal control problem (\ref{3iscdxi}).}
	\label{c3iscdxi}
\end{figure}

\autoref{c3iscdalpha} illustrates the optimal state and control trajectories for the time-optimal control problem (\ref{3iscdalpha}). The simulation uses the parameter values $\gamma = 7.0,\  \xi = 0.1,\ \delta = 3.0, \ m=1.0,\ \epsilon = 0.3$, starting from the initial point $(5,2)$ and reaching the final point $(1,4)$ in an optimal time of $2.1$ units. 

\autoref{c3iscdxi} illustrates the optimal state and control trajectories for the time-optimal control problem (\ref{3iscdxi}). The simulation uses the parameter values $\gamma = 4.0,\  \alpha = 1.0,\ \delta = 2.0, \ m=1.0,\ \epsilon = 0.5$, starting from the initial point $(5,2)$ and reaching the final point $(1,4)$ in an optimal time of $5.05$ units. 

\section{Discussions and Conclusions} \label{sec:disc}

\indent This paper studies deterministic additional food provided prey-predator system exhibiting Holling type-III functional response and intra-specific competition among predators. To begin with, we proved the positivity and boundedness of solutions. As the model could exhibit atmost $5$ real roots, we proved the condition for the existence of interior equilibria in \autoref{3iscdintcond}. The conditions for stability of various equilibria is presented in \autoref{sec:3iscdstab}. From the qualitative theory of dynamical systems, we observed that the system (\ref{3iscd}) exhibits the transcritical, saddle-node and Hopf bifurcations. We further presented a detailed study on the various interior equilibria with respect to the intra-specific competition ($\epsilon$) of predators. Due to the S-shaped behavior of the interior equilibria, two saddle-node bifurcations are observed at the two folds and this further lead to the depiction of hysteresis loop. In addition to this, we presented a detailed study on the global dynamics and consequences of providing additional food. Further, we formulated the time-optimal control problems with the objective to minimize the final time in which the system reaches the pre-defined state. Here, we considered the quality and the quantity of additional food as control variables. Using the Pontraygin maximum principle, we characterized the optimal control values. We also numerically simulated the theoretical findings and applied them in the context of pest management.

This study comprehensively studied the importance of additional food and intra-specific competition in the dynamics of the system (\ref{3iscd}). The global dynamics of the system is studied in \autoref{sec:3iscdglobaldynamics}. In this section, we observed that the provision of additional food can lead to the bistability condition. The system can converge to the pest dominant equilibrium or any of the $5$ possible interior equilibria based on the initial conditions. Taking this analysis further, it is observed from the theorem \autoref{inteqstab3iscd} that the minimum equilibria pest population is $\frac{\epsilon}{1+\frac{\epsilon}{\gamma}}$. Therefore, in the case where the stable equilibrium is only any of the interior equilibria, this work gives the least amount of stable pest population, which is directly proportional to the intra-specific competition ($\epsilon$). As intra-specific competition is due to prey scarcity, higher competition leads to the difficulty in reaching the equilibrium state. 

\autoref{SPlot3iscd} depicted the nature of various equilibria with respect to intra-specific competition. This plot depicts the bistability region where the system exhibits two stable interior equilibria. As $\epsilon$ changes with time over the region $[0.002,0.02]$, the system stabilises either in the top branch or bottom branch of the bistability region depending on the path taken by the trajectory. The dependence on the history of the trajectory plays a crucial role in the biological control of pests. As pest eradication is never stable in this system, keeping the pest at low levels so that they won't harm the crop, is the next best available strategy. Therefore, the bottom branch in the bistability region is more desirable to the top branch for pest management. Therefore, to remain in the bottom branch, it is advisable to increase natural enemies population over time rather than mass release of natural enemies for this case. Also, the sudden jump after the second saddle-node bifurcation can lead to immediate increase of pests. This sudden surge is drastic not only because of the damage the increased pest can cause to the crop but also the reversibility of damage takes time as the system should reach the first saddle-node bifurcation point. These conculsions provide valuable insights to the biocontrol strategies which is a double-edged sword. 

Some of the salient features of this work include the following. This work captures the commonly observed intra-specific competition among predators for the additional food provided prey-predator system exhibiting Holling type-III functional response. A rigorous analysis on the dynamics of this system is presented in this work. In addition to this, this paper also dealt with the novel study of the time-optimal control problems by transforming the independent variable in the control system. This work has been an initial attempt dealing with the time optimal control studies for prey-predator systems involving intra-specific competition among predators. This initial exploratory research will further lead to a more sophisticated model and rigorous analysis in the context of sensitivity and estimation of parameters, controllability and observability in the future works.

\subsection*{Financial Support: }
This research was supported by National Board of Higher Mathematics (NBHM), Government of India (GoI) under project grant - {\bf{Time Optimal Control and Bifurcation Analysis of Coupled Nonlinear Dynamical Systems with Applications to Pest Management, \\ Sanction number: (02011/11/2021NBHM(R.P)/R$\&$D II/10074).}}

\subsection*{Conflict of Interests Statement: }
The authors have no conflicts of interest to disclose.

\subsection*{Ethics Statement:} 
This research did not required ethical approval.

\subsection*{Acknowledgments}
The authors dedicate this paper to the founder chancellor of SSSIHL, Bhagawan Sri Sathya Sai Baba. The contributing author also dedicates this paper to his loving elder brother D. A. C. Prakash who still lives in his heart.

\printbibliography

@article{Adhikary2021,
	author = {Prabir Das Adhikary and Saikat Mukherjee and Bapan Ghosh},
	doi = {10.1016/j.tpb.2021.05.002},
	issn = {10960325},
	journal = {Theoretical Population Biology},
	month = {8},
	pages = {44-53},
	pmid = {34052251},
	publisher = {Academic Press Inc.},
	title = {Bifurcations and hydra effects in Bazykin's predator–prey model},
	volume = {140},
	year = {2021},
}

@article{bazykin1976structural,
	title={Structural and dynamic stability of model predator-prey systems},
	author={Bazykin, AD},
	year={1976},
	publisher={RM-76-008}
}

@article{BioControl_India,
	title={Success stories in biological control: Lessons learnt},
	author={Ballal, Chandish R},
	journal={Vantage: Journal of Thematic Analysis},
	volume={3},
	number={1},
	pages={7--20},
	year={2022}
}

@Article{CasADi,
	author = {Joel A E Andersson and Joris Gillis and Greg Horn
	and James B Rawlings and Moritz Diehl},
	title = {{CasADi} -- {A} software framework for nonlinear optimization
	and optimal control},
	journal = {Mathematical Programming Computation},
	volume = {11},
	number = {1},
	pages = {1--36},
	year = {2019},
	publisher = {Springer},
	doi = {10.1007/s12532-018-0139-4}
}

@article{howard1998gronwall,
	title={The gronwall inequality},
	author={Howard, Ralph},
	journal={lecture notes},
	year={1998}
}

@book{perko2013differential,
	title={Differential equations and dynamical systems},
	author={Perko, Lawrence},
	volume={7},
	year={2013},
	publisher={Springer Science and Business Media}
}

@book{cesari2012optimization,
	title={Optimization—theory and applications: problems with ordinary differential equations},
	author={Cesari, Lamberto},
	volume={17},
	year={2012},
	publisher={Springer Science and Business Media}
}

@article{V3EarthSystems,
  title = {Influence of Gestation Delay and the Role of Additional Food in Holling Type {{III}} Predator--Prey Systems: A Qualitative and Quantitative Investigation},
  shorttitle = {Influence of Gestation Delay and the Role of Additional Food in Holling Type {{III}} Predator--Prey Systems},
  author = {Chhetri, Bishal and Kanumoori, Deva Siva Sai Murari and Vamsi, D. K. K.},
  date = {2021-06},
  journaltitle = {Modeling Earth Systems and Environment},
  shortjournal = {Model. Earth Syst. Environ.},
  volume = {7},
  number = {2},
  pages = {897--915},
  issn = {2363-6203, 2363-6211},
  doi = {10.1007/s40808-020-01042-y}
}

@article{V3JTB,
  title = {Additional Food Supplements as a Tool for Biological Conservation of Predator-Prey Systems Involving Type {{III}} Functional Response: {{A}} Qualitative and Quantitative Investigation},
  shorttitle = {Additional Food Supplements as a Tool for Biological Conservation of Predator-Prey Systems Involving Type {{III}} Functional Response},
  author = {Srinivasu, P. D. N. and Vamsi, D. K. K. and Ananth, V. S.},
  date = {2018-10},
  journaltitle = {Journal of Theoretical Biology},
  shortjournal = {Journal of Theoretical Biology},
  volume = {455},
  pages = {303--318},
  issn = {00225193},
  doi = {10.1016/j.jtbi.2018.07.019}
}

@article{V4DEDS,
  title = {Biological {{Conservation}} of {{Living Systems}} by {{Providing Additional Food Supplements}} in the {{Presence}} of {{Inhibitory Effect}}: {{A Theoretical Study Using Predator}}--{{Prey Models}}},
  shorttitle = {Biological {{Conservation}} of {{Living Systems}} by {{Providing Additional Food Supplements}} in the {{Presence}} of {{Inhibitory Effect}}},
  author = {Srinivasu, P. D. N. and Vamsi, D. K. K. and Aditya, I.},
  date = {2018-01},
  journaltitle = {Differential Equations and Dynamical Systems},
  shortjournal = {Differ Equ Dyn Syst},
  volume = {26},
  number = {1-3},
  pages = {213--246},
  issn = {0971-3514, 0974-6870},
  doi = {10.1007/s12591-016-0344-4}
}

@article{V4Acta,
  title = {Additional {{Food Supplements}} as a {{Tool}} for {{Biological Conservation}} of {{Biosystems}} in the {{Presence}} of {{Inhibitory Effect}} of the {{Prey}}},
  author = {Vamsi, D. K. K. and Kanumoori, Deva Siva Sai Murari and Chhetri, Bishal},
  date = {2020-09},
  journaltitle = {Acta Biotheoretica},
  shortjournal = {Acta Biotheor},
  volume = {68},
  number = {3},
  pages = {321--355},
  issn = {0001-5342, 1572-8358},
  doi = {10.1007/s10441-019-09371-x}
}

\end{document}